\documentclass[letter, hyperref, fullpage, 11pt]{article}

\usepackage{cite}
\usepackage{hyperref}
\usepackage{color}

\usepackage{amsmath,amssymb,amsfonts,amsthm,enumerate}
\usepackage{mathrsfs}
\usepackage{epsf, epsfig}

\usepackage{epstopdf}

\definecolor{refkeybis}{gray}{.65}
\definecolor{labelkeybis}{gray}{.65}
{\makeatletter
\def\SK@refcolor{\color{refkeybis}}%
\def\SK@labelcolor{\color{labelkeybis}}}

\numberwithin{equation}{section} 

\newtheorem{theorem}{Theorem}[section]

\newtheorem{remark}[theorem]{Remark}

\newtheorem{claim}[theorem]{Claim}

\newcommand{\R}{\boldsymbol R}

\newcommand{\x}{\boldsymbol x}

\renewcommand{\l}{L}
\renewcommand{\L}{\boldsymbol l}
\renewcommand{\t}{\boldsymbol t}

\newcommand{\p}{\boldsymbol p}
\newcommand{\q}{\boldsymbol q}
\renewcommand{\r}{\boldsymbol r}
\newcommand{\y}{\boldsymbol y}

\newcommand{\dist}{\operatorname{dist}}
\newcommand{\diam}{\operatorname{diam}}

\newcommand{\pk}{{k}}
\newcommand{\pkm}{{k-1}}
\newcommand{\pkp}{{\pk-1}}

\newcommand{\ctk}{c_0}

\newcommand{\dtk}{\delta_{n-\pk+1} (s) }
\newcommand{\dti}{\delta_i (s)}
\newcommand{\dtip}{\delta_{i+1} (s)}
\newcommand{\rti}{\gamma_i (s)}
\newcommand{\dtkl}{\delta_{n-\pk+1}(s)}
\newcommand{\dtkp}{\delta_{n-\pk+2}(s)}
\newcommand{\dtn}{\delta_{n}(s)}

\newcommand{\rtk}{\gamma_{n-\pk+1}(s)}
\newcommand{\rtkl}{\gamma_{n-\pk+1}(s)}

\newcommand{\A}{{\bf A}}
\newcommand{\B}{{\bf B}}
\newcommand{\C}{{\bf C}}
\newcommand{\Ak}{${\bf A}_\pk \,$}
\newcommand{\Akp}{${\bf A}_{\pk-1} \,$}

\newcommand{\Bk}{${\bf B}_\pk \,$}
\newcommand{\Bkp}{${\bf B}_{\pk-1} \,$}

\newcommand{\Ck}{${\bf C}_\pk \, $}
\newcommand{\Ckp}{${\bf C}_{\pk-1}\,$}

\newcommand{\Al}{${\bf A}_\pk \, $}

\newcommand{\Alm}{${\bf A}_{\pk-1} \,$}
\newcommand{\Bl}{${\bf B}_\pk \,$}

\newcommand{\Blm}{${\bf B}_{\pk-1}\,$}
\newcommand{\Cl}{${\bf C}_\pk \, $}

\newcommand{\Clm}{${\bf C}_{\pk-1}\,$}

\newcommand{\Aone}{${\bf A}_1 \, $}
\newcommand{\Bone}{${\bf B}_{1} \,$}
\newcommand{\Cone}{${\bf C}_{1} \,$}

\newcommand\zero{\boldsymbol 0}

\newcommand\red{\color{black}}
\newcommand\blue{\color{black}}
\newcommand\green{\color{black}}

\title{{\red On supporting hyperplanes  to convex  bodies}\thanks{The
authors are pleased to acknowledge the partial support of their research
by United States' National Science Foundation grant DMS 0969962 [AF] and
Natural Sciences and Engineering
Research Council of Canada Grants 371642-09 [YHK] and 217006-08 [RJM]. \ Any opinions, findings and conclusions or recommendations
expressed in this material are those of authors and reflect the views
neither of the United States' National Science Foundation nor of the
Natural Sciences and Engineering Research Council of Canada.
\copyright 2011 by the authors.
}}
\date{\today}
\author{
Alessio Figalli\thanks{Mathematics Dept., University of Texas, Austin TX 78712 USA {\tt figalli@math.utexas.edu}},
Young-Heon Kim\thanks{Mathematics Dept., University of British Columbia, Vancouver BC Canada V6T 1Z2 {\tt yhkim@math.ubc.ca}},
and Robert J. McCann\thanks{Mathematics Dept., University of Toronto, Toronto, Ontario Canada M5S 2E4 {\tt mccann@math.toronto.edu}}
}

\begin{document}

\maketitle

\begin{abstract}
Given a convex set and an interior point close to the boundary,
we prove the existence of a supporting hyperplane whose distance to the point is controlled,
in a dimensionally quantified way, by the thickness of the convex set in the orthogonal direction.
This result has important applications in the regularity theory for Monge-Amp\`ere type equations
arising in optimal transportation. 
\end{abstract}



%
{\blue \section{Introduction} }
In this note we establish an estimate which quantifies the dimensional dependence
of the claim that corresponding to any (interior) point near the boundary of a convex set,
is a supporting hyperplane much closer
than the thickness of the set in the orthogonal direction.
The main interest of {\blue our estimate (see Theorem~\ref{T:ratio})} is 
that
it allows us \cite{FKM} to extend --- for the first time --- 
{\red a H\"older continuity result of Caffarelli~\cite{C, caffC1a} concerning}
optimal transportation of bounded measurable densities
from the quadratic cost function
of Brenier \cite{B},
to the more general cost functions considered by Trudinger and Wang~\cite{TW}.
{\blue Caffarelli's regularity result has well-known connections to convex geometry (see \cite{gutierrez}), and the importance of its extension to more general optimal transport problem is highlighted in \cite{V2}.} 

{\blue Our}
theorem below is purely geometric, elementary to state,  and non-trivial to prove; it may well be of
independent interest. To emphasize this possibility, and the fact that it does not rely on
any auxiliary concepts arising from the intended application \cite{FKM},  we establish it in this
separate manuscript.
By so doing, we hope to ensure its accessibility to convex geometers who may
have no interest in optimal transportation,  as well as to its primary target audience, which
consists of researchers interested in the regularity of optimal mappings (or equivalently,
of degenerate elliptic solutions to the associated Monge-Amp\`ere type equations).

Let us start by recalling that
a {\em convex body} $\tilde S$ {\blue in the $n$-dimensional Euclidean space} $\R^n$ refers to a compact convex set with non-empty
interior. A well-known result of Fritz John \cite{John}, {\blue often called John's Lemma}, shows every convex body  can be translated
so that it contains an ellipsoid $E$ centered at the origin,  whose dilation by factor $n$
contains the translated copy $S$ of $\tilde S$:
\begin{equation}\label{E:well-centered}
E \subset S \subset n E.
\end{equation}
The constant $n$ is shown to be sharp by the standard simplex.
After this translation,  i.e.~when \eqref{E:well-centered} holds, we call $S$ {\em well-centered}.
{\blue We restrict our discussion to only well-centered convex bodies, but this does not cost any generality.}

For any point $\y \in S$ near the boundary of a well-centered convex body,
we claim it is possible to find
a direction in which the boundary of $S$ is much
closer than the thickness of $S$ in the same direction.  More precisely,
we claim it is possible to find
a line $\l$ through the origin whose intersection with
$S$ is large relative to the distance of $\y$ to a hyperplane outside of $S$ and orthogonal
to $\l$.  Here {\em orthogonal} refers to the ambient Euclidean inner product,  so that the
two distances being compared are measured along line segments parallel to $\l$.
The following theorem quantifies the dependence of their ratio
on the proximity of $\y$ to the boundary,  as reflected in the
degenerating factor $s^{1/2^{n-1}}$ in \eqref{E:key reformulation}
below. When $\y$ approaches the
boundary of $S$, the ratio of the two distances becomes more and more exaggerated,
algebraically fast with respect to the separation of $\y$ from the boundary
(but whose algebraic power decays exponentially fast in high dimensions).
Here $s$ measures the separation of $\y$ from the boundary in the Minkowski gauge of $S$
--- which of course is equivalent to any other norm on $\R^n$.  In what follows however,
`$\dist$', `$\diam$' (and orthogonality) always refer to distance and diameter
with respect to the Euclidean norm.  The application \cite{FKM} requires only
the special case {\green $s_0=1/(2n)$}.

\bigskip
\begin{theorem}
\label{T:ratio}
Let $S \subset \R^n$ be a well-centered convex body, meaning \eqref{E:well-centered}
holds for some ellipsoid $E$ centered at the origin. Fix $ 0 \le s \le s_0 <1$.
{\green For each $\y \in (1-s)\partial S$ there exists an hyperplane $P$
supporting $S$ such that 
 \begin{equation}\label{E:key reformulation}
\dist(\y, P) \le c(n,s_0) s^{ 1/2^{n-1}} \diam(P^\perp \cap S).
 \end{equation}
Here $P^\perp$ denotes
the (unique) line orthogonal to $P$ passing through the origin, and
$c(n,s_0)$ is a constant depending only on $n$ and $s_0$, namely
$c(n,s_0) = n^{3/2} (n-\frac{1}{2})
\Big( \frac{1+(s_0)^{1/2^n}}{1-(s_0)^{1/2^n}} \Big)^{n-1}
$. }
\end{theorem}
\begin{remark}
{\blue For $n=1$, the} constant $c(1,s_0)$ is sharp and \eqref{E:key reformulation} becomes an obvious equality in that case;
we have not investigated sharpness of $c(n,s_0)$ or of the power $2^{1-n}$ in higher dimensions.
\end{remark}


{\blue  The key point of the estimate \eqref{E:key reformulation} is that the ratio $\dist(\y, P)/\diam(P^\perp \cap S)$}
{\green goes to $0$ as $s \to 0$ in a ``uniform way'', independent of the shape of $S$.}
Observe that if $s=0$ we can choose {\blue  $P$ to support} $S$ at $\y \in \partial S$,
but for $s>0$ it is less obvious how to choose {\blue $P$ (and hence $L$).} 
{\blue The difficulty for proving this estimate is on the arbitrariness of the convex body $S$. For example, if $S$ is
the round ball, then the estimate~\eqref{E:key reformulation} becomes trivial}
{\green (indeed, one may even replace $c(n,s_0) s^{ 1/2^{n-1}}$ with $s$, and the supporting hyperplane $P$
shall be chosen to be orthogonal to the vector $\y$).} {\blue For a general convex bodiy $S$,
there are three natural ways to try to generalize such choice: (i) to choose $P$ orthogonal to $\y$, (ii) to choose $P$ supporting
$S$ at the intersection of the half line $\overrightarrow{\zero\y}$ with $\partial S$, or (iii) to choose $P$ closest to $\y$. 
However, in all these three cases it is not hard to find counterexamples (some family of degenerating thin convex bodies)
showing such rather natural choices of hyperplanes do not work, namely, {\em not} yielding
{\green a uniform convergence to $0$ of the ratio $\dist(\y, P)/\diam(P^\perp \cap S)$ as $s\to 0$}.   To prove Theorem~\ref{T:ratio}  we find an algorithm which allows to choose appropriate $P$ by an inductive recursion, reducing the dimension of the task confronted at each step.}


{\blue 
One of the reasons why the estimate \eqref{E:key reformulation} is nontrivial is that the set of convex bodies is not compact. {\red A common} and powerful way to deal with such non-compactness is to use John's Lemma~\cite{John} (see \eqref{E:well-centered}), to renormalize the convex bodies via affine maps so that the resulting shapes become roughly close to the round ball (with a uniformly bounded scale factor). This way, one can easily derive some estimates for quantities that are affine invariant. One such example is the classical Alexandrov estimates for the Monge-Amp\`ere measure associated to sections of convex functions: see, for example \cite{gutierrez}.  
However, in our case the inequality \eqref{E:key reformulation} involves orthogonality with respect to the {\em fixed} Euclidean norm, which is not affine invariant. Therefore, we cannot derive the estimate \eqref{E:key reformulation}  by applying John's Lemma. 

The remainder of this paper is devoted to the proof of Theorem~\ref{T:ratio}, which is completely elementary though quite nontrivial.}
%

{\blue 
\section{Proof of Theorem~\ref{T:ratio}}
}

Let $S \subset \R^n$ be a well-centered convex body.
{\blue Let the positive numbers $a^1, \cdots, a^n \in \R_+$ denote the lengths of the principal
semi-axes of the inner ellipsoid $E$ of Fritz John \eqref{E:well-centered}. 
One can regard these $a^i$'s as the coordinate components of the vector ${\boldsymbol a}= (a^1, \cdots , a^n)$.
In the following, superscripts will be used to denote such coordinate components for vectors, and
for all the other cases they will mean powers {\green (we believe this should not create either confusion or ambiguity)}.}

Use these principal axes to choose coordinates, with the origin $\zero$ at the center of $E$.
We still have the freedom to choose the order in which these axes are enumerated,  which
 we shall exploit {\blue especially} at \eqref{E:j=k}.
In these coordinates $E$ is represented as
\begin{align*}
E = \biggl\{ \x = (x^1, \cdots, x^n ) \in \R^n  \ | \ \sum_{i=1}^n \Big{(}\frac{x^i}{a^i}\Big{)}^2 \le 1 \biggr\}.
\end{align*}
The rectangle
\begin{align*}
R_n &= \{ \x = (x^1, \cdots, x^n ) \in \R^n  \ | \  | x^i | \le n a^i ,  \ i =1, \cdots, n \}
\end{align*}
circumscribed around the outer ellipse $nE$ will also play a crucial role.
Observe $\frac{1}{n^{3/2}}R_n \subset S \subset R_n$.
In particular, $S$ is comparable (in size and shape) to  $R_n$.

\bigskip
{\blue \subsection{Initial step in the recursive algorithm}\label{SS:initial step}}
{\green Fix $s_0 \in (0,1),$ and given $\y \in (1-s) \partial S$, $0 \le s \le s_0$,}
let $\p_n$ be the intersection of the half line $\overrightarrow{\zero\y}$ with $\partial S$.
We pick a tangent hyperplane $P_n$ (which may not be unique)  to $S$ at $\p_n$. Using
similar triangles, we deduce that
\begin{align*}
\frac{\dist(\y, P_n)}{\dist(\zero,P_n)} = s.
\end{align*}

{\blue 
\subsubsection*{Two alternatives}
We consider the following two exclusive cases:

 {\bf Favorable case:}}
Suppose we are lucky enough that
\begin{align}\label{E:trivial case}
\frac{\diam(P_n^\perp\cap R_n)}{2\dist(\zero, P_n)} \ge s^{1/2} .
\end{align}
Then, the choice $P=P_n$ leads to the desired result {\blue \eqref{E:key reformulation}} since
$\diam (P_n^\perp \cap R_n) \le n^{3/2} \diam (P_n^\perp \cap S)$, thus
\begin{align*}
\frac{\dist(\y, P_n)}{\diam (P_n^\perp\cap S)}
& \le n^{3/2}  \frac{\dist(\y, P_n)}{\dist(\zero, P_n)} \frac{\dist(\zero, P_n)}{\diam(P_n^\perp\cap R_n)}\\
& \le {n^{3/2}} \frac{s}{2s^{1/2}} \\
& = \frac{n^{3/2}}2 s^{1/2}.
\end{align*}

 {\blue {\bf Unfavorable case:}}
If the convex {\blue body}  $S$ is very thin, or equivalently if the outer rectangle $R_n$ is very thin, then the ratio
{\blue ${\diam(P_n^\perp\cap R_n)}/{2\dist(\zero, P_n)}$} can be much smaller than $s^{1/2}$,
in which case \eqref{E:trivial case} fails.
For such situations, we now describe a recursive algorithm which shows that whenever
\eqref{E:trivial case} is violated, {\blue after at most $(n-1)$-steps} it is possible to find {\blue an alternative} hyperplane $P$ (generally different
from $P_n$) which fulfils the desired conclusion \eqref{E:key reformulation}.

\bigskip
{\blue 

\subsection{Notation in the recursive algorithm}

The basic idea of the following recursive algorithm is to repeat the previous {\bf two alternatives} in the inductive steps,
with decreasing dimension. Since this is a  finite dimensional situation, such algorithm should terminate,
and we show it does so yielding the desired result  \eqref{E:key reformulation}.
One of the key points of the argument is to choose the {\em right} geometric configuration.
This requires in particular some careful choice of the terms $\rti, \dti$, and $\ctk$, as we fix the notation {\green below}.

}
Let $i,\pk \in \{1,\ldots, n\}$ and  $s \in [0,s_0]$.
{\blue We define
\begin{align*}
\rti & := s^{1/2^{i}}; \\
\dti &:= (2i-1) s^{1/2^{i-1}}.
\end{align*}
These satisfy the following key relations:}
\begin{align}\nonumber
\dtip &\ge  \dti + 2 \rti ; \\\label{E:dtk rtk ratio}
\frac{\dti}{\rti}  & =  (2i -1) \rti .
\end{align}


We use coordinates {\blue $( x^{1}, \cdots, x^{\pk+1})$ on $\R^{\pk +1 }$.}
Define the projections  $\pi_\pk : \R^{\pk+1} \to \R^{\pk}$ by
\begin{align*}
& \pi_\pk (x^{1}, \cdots, x^\pk, x^{\pk+1}) = (x^{1}, \cdots, x^\pk).
\end{align*}
{\blue Observe that each $\pi_\pk$ is determined by the choice of the coordinate axis for $x^{\pk+1}$ for $\R^{\pk+1}$, and such choice will be made individually at each step of the recursive algorithm. This is an important point to remember throughout the proof.}
For $\pk<n$, define the rectangles $R_\pk$ in $\R^\pk$ inductively as dilated projections of
$R_n$:
\begin{align*}
R_\pk = \ctk \pi_\pk (R_{\pk+1}).
\end{align*}
Here,  the dilation factor $\ctk$ (with respect to the origin $\zero$) is given by
\begin{align}\label{E:cn}
\ctk &:= \max_{0 \le s \le s_0; \,  1 \le i \le n} \Big{[} 1 + 2 \frac{\rti}{1-\rti}\Big{]}  \\
&= \frac{1+(s_0)^{1/2^n}}{1-(s_0)^{1/2^n}},
\end{align}
where the monotone dependence of $\rti$ on both $s \le s_0$ and $i\le n$ has been used.
These rectangles can also be written as
 \begin{align}
\label{eq:Rk}
R_\pk =\{ \x = (x^{1}, \cdots, x^\pk ) \in \R^{\pk}  \ | \  | x^i | \le \ctk^{n-k} {\blue n a^i} , \ \ 1 \le i \le \pk \}.
\end{align} 
Let $Q^\pm_\pk$ (and $Q^0_\pk$) be the parallel hyperplanes
in $\R^\pk$ which form the boundary of (and bisect) $R_\pk$
orthogonally to the $x^\pk$-axis:
\begin{align*}
Q^\pm_\pk &= \{\x = (x^1, \cdots, x^\pk) \in \R^{\pk} \ | \  x^\pk = \pm \ctk^{n-\pk} {\blue n  a^\pk } \};\\
Q^0_\pk & =  \{\x = (x^1, \cdots, x^\pk) \in \R^{\pk} \ | \  x^\pk =0 \}.\\
\end{align*}
Set $\y_n=\y=(y^1,\ldots,y^n)$ and define its projections recursively
 \begin{align*}
\y_\pk := \pi_{\pk} \circ \pi_{\pk+1} \circ \cdots \circ \pi_{n-1} (\y )=(y^1,\ldots,y^\pk).
\end{align*}
Since $\ctk \ge 1$ it is clear that  $\y_\pk \in R_\pk$.

In the following we will define some hyperplanes $P_\pk \subset \R^\pk$ inductively.
For such  a hyperplane $P_\pk \subset \R^{\pk}$ use  $\tilde P_\pk$ to denote the extension of
$P_\pk$ to the hyperplane in $\R^n$ parallel to the $x^{k+1}$- through $x^n$-axes,
i.e., which satisfies
 \begin{align*}
 \pi_{\pk} \circ \cdots \circ \pi_{n-1} (\tilde P_\pk ) = P_\pk.
\end{align*}

{\blue \subsection{ The structure of the recursive algorithm}
Recall the point $\p_n =  \overrightarrow{\zero\y} \cap \partial S$ and the supporting hyperplane $P_n$, with $\p_n \in P_n$, which are given in the initial step of the recursive algorithm (Section~\ref{SS:initial step}). 
}
Let us  describe the recursion in which we use $(\p_\pk, P_\pk)$ to define
$(\p_{\pk-1}$, $P_{\pk-1})$. 
{\blue  We first list three conditions that are required at each step of the recursive algorithm.}
 For the $i^{th}$ step (here $i=n-\pk+1$) assume that
\begin{align*}
&\hbox{\Ak:}    \qquad
\p_\pk \in  P_\pk \cap R_\pk ,  \qquad  \y_\pk \in [\zero, \p_\pk];\\
&\hbox{\Bk:}  \qquad
  \frac{\dist(\y_\pk, P_\pk)}{\dist (\zero, P_\pk) } \le \dtkl ;   \\
& \hbox{\Ck:}  \qquad \hbox{ $\tilde P_\pk$ does not intersect the interior of $S$}.
\end{align*}
Notice that {\blue for $\pk=n$, either the favorable situation \eqref{E:trivial case} holds, in which case there is nothing further to prove (and so no need to proceed to the next step),
or else $\A_n,  \B_n, \C_n$  are satisfied  by our initial choice of
$\p_n$ and $P_n$.
}

Starting from $\pk=n$,  we shall decrease $\pk$ one step at a time until the algorithm
terminates.  Whether or not the recursion terminates at a given value of $\pk$ is determined
by the following dichotomy:
\begin{align}\label{E:good ratio k}
&\hbox{\bf Case I (favorable case):}  & \frac{\diam(P_\pk^\perp\cap R_\pk)}{2\dist(\zero, P_\pk)} & \ge \rtkl;    \\ \label{E:ratio assumption}
& \hbox{\bf Case II (unfavorable case):} & \frac{\diam(P_\pk^\perp\cap R_\pk)}{2\dist(\zero, P_\pk)} & < \rtkl.
\end{align}
If {\bf Case I} holds for some value of $\pk$ we shall discover we are in a favorable situation
--- analogous to \eqref{E:trivial case} --- which allows us to terminate the recursion and
obtain the desired result {\blue \eqref{E:key reformulation}}.  On the other hand,  if {\bf Case II} holds for the given value of
$\pk$,  we shall see we can use $(\p_\pk, P_\pk)$ satisfying the inductive hypotheses
\Al, \Bl and \Cl, {\blue and the condition \eqref{E:ratio assumption}}, to construct $(\p_{\pk-1},P_{\pk-1})$ satisfying \Alm, \Blm and \Clm.  We
then decrease $\pk$ and proceed to the next step of the recursion.  In the worst case
the recursion continues until $\pk=1$, and we find $(p_1,P_1)$ satisfying
\Aone, \Bone and \Cone.  In this case we show in the last section below that
the desired result {\blue \eqref{E:key reformulation}} can again be obtained, to complete the proof of the theorem. 

\bigskip

{\blue \subsection{Case I, \eqref{E:good ratio k} holds for some $\pk \ge 2$:
the recursion terminates with the desired result.}
}
As soon as we reach {\blue some} $\pk \ge 2$ for which {\blue the condition} \eqref{E:good ratio k} holds, we
stop the recursion. We now show in this case the desired result {\blue \eqref{E:key reformulation}}  follows. {\blue Here,  the assumptions \Bk and \Ck are crucial.}
Since, by the construction of {\blue $\y_\pk$ and} $\tilde P_\pk$, $\dist(\y, \tilde P_\pk) = \dist(\y_\pk, P_\pk)$ and
$\dist(\zero, \tilde P_\pk)= \dist(\zero, P_\pk)$, recalling \eqref{eq:Rk} we get
\begin{align}\label{E:ratio for L k}
\frac{\dist (\y, \tilde P_\pk)}{\ctk^{n-\pk} \diam(\tilde P_\pk^\perp\cap R_n)}
& =  \frac{\dist(\y_\pk, P_\pk)}{\dist (\zero, P_\pk) } \frac{\dist(\zero, P_\pk)}{ \diam(P_\pk^\perp\cap R_\pk)}\cr
&\le   \frac{\dtkl}{2\rtkl}  \qquad \hbox{(by \eqref{E:good ratio k} and  \Bk)}.
\end{align}
Let $\tilde H_\pk$ be the half-space containing $\zero$,
with $\partial \tilde H_\pk  = \tilde P_\pk $.
Notice that $S \subset \tilde H_\pk$ by assumption \Ck. Thus, translating $\tilde P_\pk$
toward $S$, one can find a hyperplane $P$ {\blue supporting $S$,} which is parallel to $ \tilde P_\pk$.
Since $\dist(\y, P) \le \dist(\y, \tilde  P_\pk)$ and
$\diam(P^\perp\cap S) = \diam (\tilde  P_\pk^\perp\cap S) \ge n^{-3/2}\diam (\tilde  P_\pk^\perp\cap R_n) $,
from \eqref{E:dtk rtk ratio} and \eqref{E:ratio for L k} we have
\begin{align*}
\frac{\dist(\y, P)}{\diam(P^\perp\cap S)}
& \le n^{3/2}  \ctk^{n-\pk}    \frac{\dtkl}{2\rtkl}\\
& = n^{3/2}  \ctk^{n-\pk}   (n -k+{\textstyle \frac{1}{2}})  {\blue s^{1/2^{n-k+1}}} \\
& \le n^{3/2}  \ctk^{n-1} (n -{\textstyle \frac{3}{2}})  s^{1/2^{n-1}}
\end{align*}
(recall that $k \geq 2$), which gives the desired result {\blue \eqref{E:key reformulation}}.


\bigskip

\subsection{Case II, \eqref{E:ratio assumption} holds for $2 \le \pk \le n$:
the recursion continues.}

A remark before we proceed: in the following argument, we assume that any claimed intersections
between affine subspaces such as lines and (hyper-)planes actually exist and have the
expected (i.e. generic) dimension.  This costs no generality for two reasons:
\begin{itemize}
\item  To avoid parallelism we can perturb if necessary (i.e. rotate and/or translate slightly) the affine subspaces.
\item  We will obtain estimates which are not sensitive to small perturbations,
so the estimates also hold without the perturbation.
Moreover, the obtained bounds then imply that the claimed intersections do indeed exist.
\end{itemize}
Similarly, we can also avoid, if necessary, the cases where some lengths and/or distances
degenerate to zero.

{\blue 
\subsubsection*{Some preliminaries and the construction of $P_{\pk-1}$}
}
To define $P_{\pk-1}$ and $\p_{\pk-1}$ we set-up preliminaries.
Let us first consider the point $\r_{\pk}={\blue (r_{\pk}^1, \cdots, r_{\pk}^k) \in \R^k }$ defined as the closest point on $P_\pk$ to the origin, so that $ \dist (\zero, \r_{\pk}) = \dist (\zero, P_\pk)$.
Note $\r_{\pk}$ is outside the rectangle $R_\pk$, since otherwise,
$$
\frac{\diam(P_\pk^\perp\cap R_\pk )}{2\dist(\zero, P_\pk)} \geq 1
$$
contradicting assumption \eqref{E:ratio assumption}.
Let
$\r^+_{\pk} {\blue =((\r^+_{\pk})^1, \cdots, (\r^+_{\pk})^k) }=  [0, \r_{\pk}] \cap \partial R_\pk$
denote the intersection point of the ray through $\r_{\pk}$ with whichever of the {\blue $2k$} faces
of this rectangle it {\blue intersects.} 
Without loss of generality,  suppose the axes are enumerated
so that the intersection occurs on the face of $R_\pk$ contained in $Q^+_\pk$. {\blue (Observe that this choice of coordinates affects
the definition of $\pi_\pkm$.)} {\red Then,  
\begin{align}\label{E:j=k}
(\r^+_{\pk})^\pk = \ctk^{n-\pk} a^\pk > 0
\end{align}
holds.}
Because
$\dist(\zero, \r^+_{\pk}) =  \frac{1}{2}\diam (P_\pk^\perp\cap R_\pk)$ and  $\dist(\zero, \r_{\pk}) = \dist(\zero, P_\pk)$, we have
 \begin{align}\label{E:similarity}
 \frac{\dist(\zero, \r^+_{\pk})}{ \dist(\zero, \r_{\pk})}
 =\frac{\diam (P_\pk^\perp\cap R_\pk)}{ 2 \dist(\zero, P_\pk)}.
 \end{align}
Now, define the hyperplane $P_{\pkm}$ in $\R^{\pk-1}$ by
 \begin{align}\label{E:choose L}
P_{\pkm} &:= \pi_{\pkm} (P_\pk \cap Q^-_\pk).
\end{align}

\subsubsection*{Verification of \Ckp}
Before proceeding further, let us verify that \Ckp\ follows from \Ck\ as a consequence.

Since $r_\pk^\pk=\sqrt{|\r_\pk|^2 -|\pi_\pkp(\r_\pk)|^2}$,
by a simple geometric argument (see Figure 1) the construction above yields $\r_\pkp = \lambda \pi_\pkp(\r_\pk)$
(recall that $\r_\pkp$ is defined as the closest point on $P_{\pk -1}$ to the origin),
with $\lambda = (|\r_\pk|^2 + \ctk^{n-\pk} a^\pk r_\pk^\pk)/|\pi_\pkp(\r_\pk)|^2>1$.

\begin{figure}
\centerline{\epsfysize=2.2truein\epsfbox{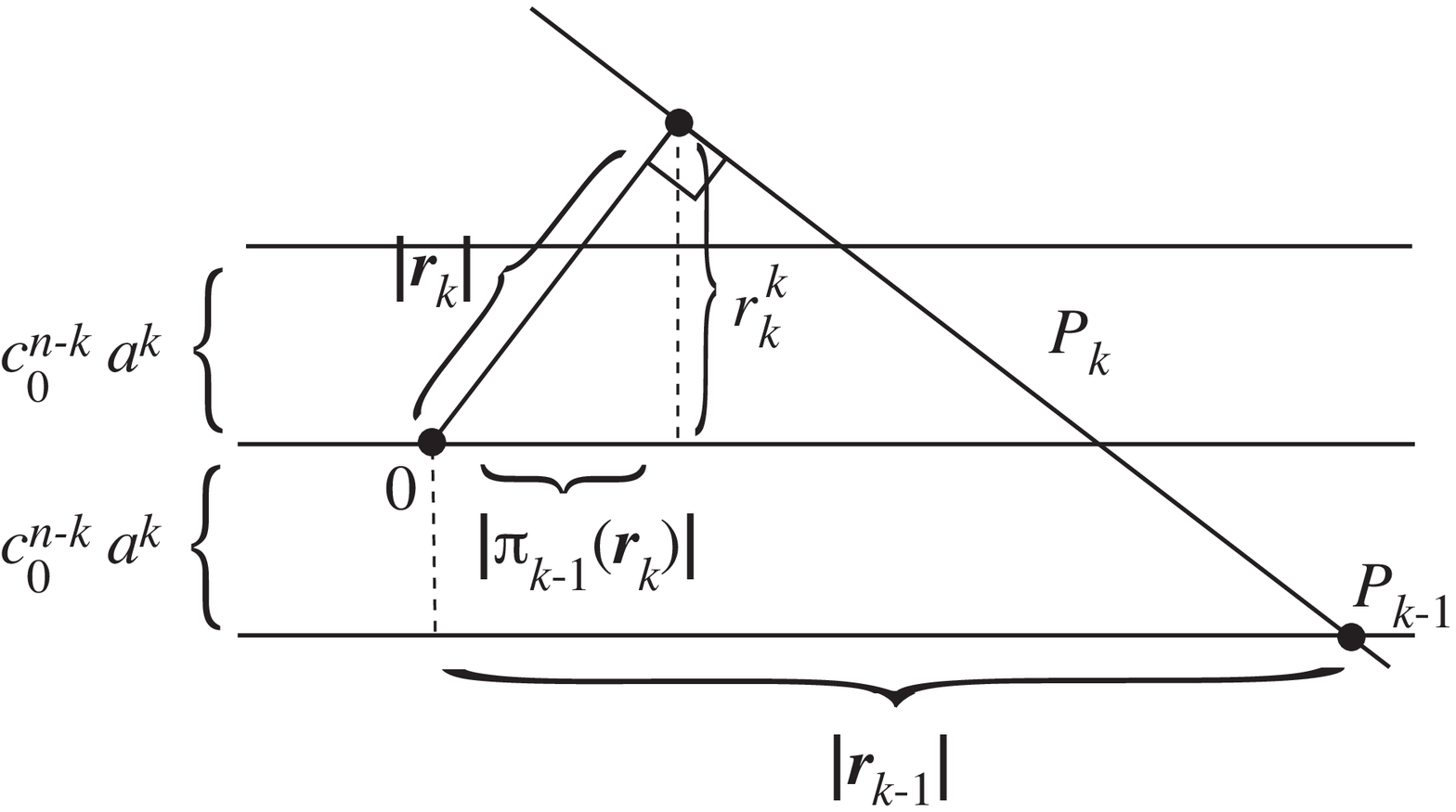}}
\caption{By a simple argument based on similar triangles, one can easily compute $|\pi_\pkp(\r_\pk)|$ in terms of
$|\r_\pk|$, $\ctk^{n-\pk} a^\pk$, $r_\pk^\pk$, and $|\pi_\pkp(\r_\pk)|$.}
\label{fig1}
\end{figure}

For each $\x=(x^1,\ldots,x^n)=:\x_n$ from the interior 
of $S$, let
$\x_\pk := \pi_\pk \circ \pi_{\pk+1} \ldots \circ \pi_{n-1}(\x) = (x^1,\ldots,x^\pk)$.
{\blue To verify \Ckp we need to show that $\x_\pkp \cdot \r_{\pkp} \le |\r_\pkp|^2$. Without loss of generality assume
that $\x_\pkp \cdot \r_{\pkp}>0$.}
Hence, {\blue since $\lambda >1$,}
\begin{align*}
\x_\pkp \cdot \r_{\pkp}
&< \lambda (\x_\pk\cdot\r_\pk + a^k r^\pk_\pk)\\
&< \lambda (|\r_{\pk}|^2  + a^k r^\pk_\pk) \\
&\leq |\r_{\pkp}|^2,
\end{align*}
where the first inequality follows from $x^k> -a^k$, the second from \Ck,
and the third from $\ctk\geq 1$ and the definition of $\lambda$.
This yields \Ckp\ as desired. 

\bigskip
{\blue  \subsubsection*{Construction of $\p_\pkp$} }
We will now find
$\p_\pkp \in \R^{\pkp}$ so that \Akp and \Bkp are satisfied.
To define $\p_\pkp$, consider the two-dimensional plane $T_\pk \subset \R^{\pk}$,
generated by the $x^\pk$-axis and the half-line $\overrightarrow{\zero\p_\pk}$
(which is the same as $\overrightarrow{\zero\y_\pk}$).
(We perturb $\p_\pk$ slightly if necessary to ensure it does not lie on the $x^k$-axis.)
Since $\dim T_\pk + \dim P_\pk = \pk+1$, the affine intersection
$\l_\pk := T_\pk \cap P_\pk \subset \R^\pk$ contains at least a line;
it contains at most a line since
$\zero \in T_\pk \setminus P_\pk$.  Notice that
the line $\l_\pk$ passes through the point $\p_\pk$
and the hyperplane $Q^-_\pk$, as in Figure 2.

\begin{figure}
\centerline{\epsfysize=3.3truein\epsfbox{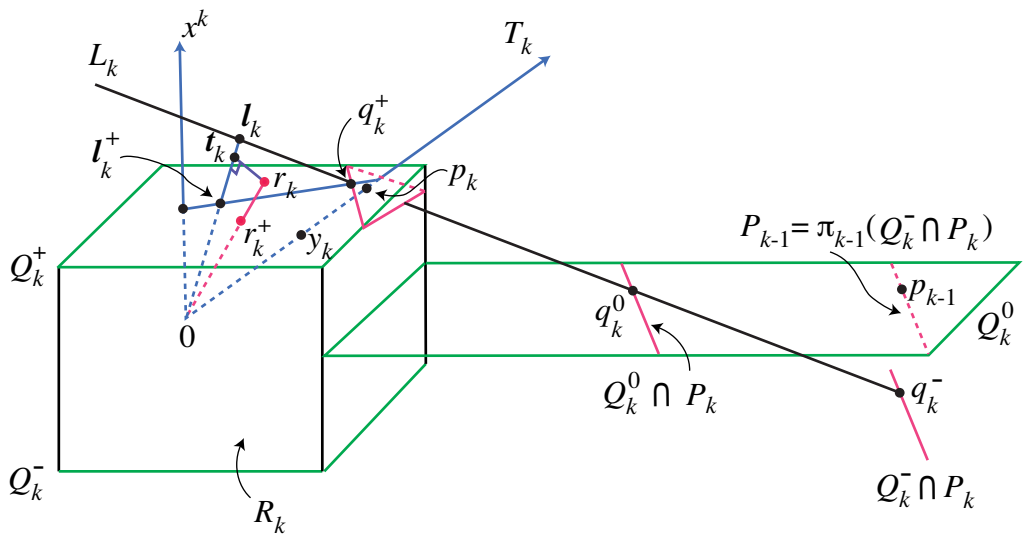}}
\caption{The geometric argument behind the construction of $\p_\pkp$. Observe that, for $s$ small, this figure (as well
as the other figures) is not very ``realistic'', as $R_{\pk}$
should be very thin in the ``horizontal'' directions and $P_{\pk-1}$ very close to $R_{\pk}$.
However, even if the proportions are not respected, this figure should help {\blue the reader {\green to} follow}
the argument described below.
{\blue Note that one can construct convex sets for which $\y_\pk$ and $\p_\pk$ may lie below $Q^0_\pk$.} {\green However,
although the picture is just indicative, our proof is purely analytic and works independently of the position of $\y_\pk$ and $\p_\pk$ with respect to $Q^0_\pk$.}}
\label{fig2}
\end{figure}

We define 
\begin{align}\label{E:choose p}
\p_\pkp := \pi_{\pk-1} ( \l_\pk \cap Q^-_\pk).
\end{align}
Notice that
\begin{align}\label{E:Akp less}
\p_\pkp \in P_{\pk-1},  \qquad  \y_\pkp \in [\zero, \p_\pkp].
\end{align}
In particular,  to verify \Akp we only need to check $\p_\pkp \in R_\pkp$.

\bigskip
{\blue \subsubsection*{Preparation before verifying \Akp and \Bkp }}
To verify \Akp and \Bkp we first find a few relevant points {\blue on the plane $T_\pk$.
What we are going to describe is summarized in Figure 2.}

The line  $\l_\pk$ intersects with the three parallel hyperplanes
$Q^0_\pk, Q^\pm_\pk \subset \R^{\pk}$.
Denote  the three intersection points by
\begin{align*}
\q^+_\pk &:= \l_\pk \cap  Q^+_\pk;\\
\q^0_\pk &:=  \l_\pk \cap Q^0_\pk ;\\
\q^-_\pk &:= \l_\pk \cap  Q^-_\pk.
\end{align*}
Notice that $\p_\pk \in [\q^+_\pk, \q^-_\pk]$ and  $\pi_{\pk-1} (\q^-_\pk)  = \p_\pkp$.

Let $\L_{\pk}$ denote the closest point on $\l_\pk$ to the origin $\zero$, and
let $\t_{\pk} \in T_\pk$ be the orthogonal projection of $\r_{\pk}$ on the plane $T_\pk$.
Notice  that, {\green since $\r_{\pk}$ is the closest point in $P_\pk$ to $\zero$
and $\l_\pk \subset P_\pk$, the orthogonal projection of the ray $\overrightarrow{\zero \r_\pk}$ to $T_\pk$ is the ray
$\overrightarrow{\zero \L_{\pk}}$ and 
$|\r_\pk| \leq |\L_{\pk}|$ (to see this, one may consider the plane passing through
$\zero$, $\r_\pk$, and $\L_\pk$, and observe that it cuts $\l_\pk$ orthogonally)}. In particular,
the point $\t_{\pk}$ belongs to the line segment  $[\zero, \L_{\pk}]$.
Moreover, since $T_k$ contains the $x^k$-axis, $t_\pk^\pk = r_\pk^\pk $. Hence,
by our assumption \eqref{E:j=k},
$\t_{\pk}$ belongs to the region over $Q^+_\pk$, namely
$t_\pk^\pk > \ctk^{n-k} a^\pk$, as (therefore) does $\L_{\pk}$.
Consider the point
\begin{align*}
\L^+_{\pk} : = [\zero, \L_{\pk}] \cap Q^+_\pk
\end{align*}
with $(\L^+_{\pk})^\pk = (\r^+_{\pk})^\pk = \ctk^{n-k} a^\pk$.  Since (as we observed above)
$t_\pk^\pk = r_\pk^\pk $, the triangles $\triangle(\t_\pk,\zero,\r_\pk)$ and
$\triangle(\L^+_\pk,\zero,\r^+_\pk)$ are similar. Thus
\begin{align}\label{E:bound dot q l}
 \frac{\dist(\zero, \L^+_{\pk})}{ \dist(\zero, \L_{\pk})}
 \le \frac{\dist(\zero, \L^+_{\pk})}{ \dist(\zero, \t_{\pk})}
 = \frac{\dist(\zero, \r^+_{\pk})}{ \dist(\zero, \r_{\pk})}
 \le  \rtk .
\end{align}
Here, the last inequality follows from \eqref{E:similarity} and \eqref{E:ratio assumption}.

\begin{figure}
\centerline{\epsfysize=1.7truein\epsfbox{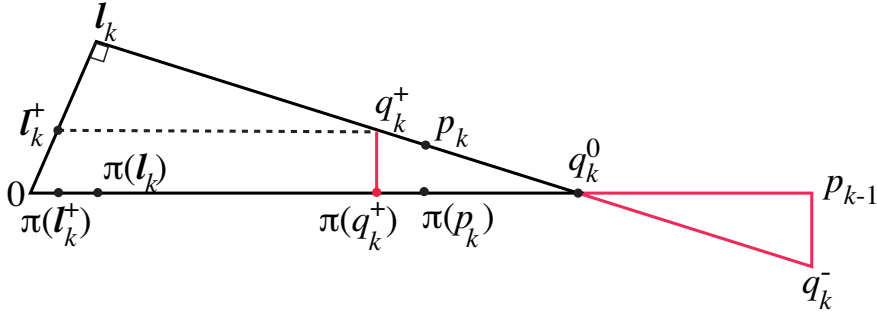}}
\caption{Some useful similitudes. {\blue Note that, {\green analogously to what observed in the comment to Figure 2}, it maybe also possible that $\p_\pk$ (resp, $\pi_{\pk-1} (\p_\pk)$) be located between $\q_\pk^0$ and $\q^-_\pk$ (resp,  $\p_{\pk-1}$).} }
\label{fig3}
\end{figure}

Now the triangle $\triangle(\zero, \L_{\pk}, \q^0_\pk)$
formed by the three points
$\zero$, $\L_{\pk}$ and $\q^0_\pk$ has a right angle at $\L_{\pk}$.  This entire triangle
projects to a line segment in $\R^\pkp$, with the projection $\pi_\pkp$ preserving the
order of points and ratios of distances along the edges of $\triangle(\zero, \L_{\pk}, \q^0_\pk)$
as in Figure 3 --- a fact we shall continue to use subsequently.

Similarity to
$\triangle (\L^+_\pk,\L_\pk,\q^+_\pk)$ combines with
\eqref{E:bound dot q l} to yield
\begin{align}\label{E:similarity l}
\frac{\dist(\pi_{\pk-1} (\q^0_\pk), \pi_{\pk-1} (\q^+_\pk))}{ \dist (\pi_{\pk-1} (\q^0_\pk),  \pi_{\pk-1} (\L_{\pk}))}
= \frac{ \dist (\q^0_\pk, \q^+_\pk)}{ \dist ( \q^0_\pk, \L_{\pk})}
=\frac{\dist(\zero, \L^+_{\pk})}{ \dist (\zero, \L_{\pk})} \le \rtk .
 \end{align}

 Now we are ready to verify \Akp and \Bkp.
\bigskip
{\blue 
\subsubsection*{Verification of \Akp}
 }
\begin{claim}\label{C:Akp}
Definitions \eqref{E:choose L}--\eqref{E:choose p} imply \Akp.
\end{claim}
\begin{proof}
By \eqref{E:Akp less} it is enough to show that $\p_\pkp \in R_\pkp$.
This should be clear from the geometric construction.
Here the factor $\ctk$ in the definition $ R_\pkp= \ctk \pi_{\pk-1} (R_\pk)$ plays a crucial role.
To give details, first note that \Ak\ implies
\begin{align*}
\pi_{\pk-1} (\p_\pk) \in \pi_{\pk-1} (R_\pk).
\end{align*}
From $\p_\pk \in [\q^+_\pk, \q^-_\pk]$ and $\p_\pkp = \pi_{\pk-1} (\q^-_\pk)$, we have
\begin{align*}
\dist(\zero, \p_\pkp)
& = \dist(\zero, \pi_{\pk-1} (\p_\pk) ) + \dist (\pi_{\pk-1} (\p_\pk),  \p_\pkp)\\
& \le \dist(\zero, \pi_{\pk-1} (\p_\pk) ) + \dist (\pi_{\pk-1} (\q^+_\pk),  \p_\pkp)\\
&  = \dist(\zero, \pi_{\pk-1} (\p_\pk) ) + 2\dist (\pi_{\pk-1} (\q^+_\pk), \pi_{\pk-1} (\q^0_\pk )) .
\end{align*}
Here, to bound the last line, observe that from \eqref{E:similarity l},
 \begin{align*}
 \dist (\pi_{\pk-1} (\q^+_\pk), \pi_{\pk-1} (\q^0_\pk)  )&  \le \frac{\rtk}{1-\rtk}
 \dist (\pi_{\pk-1} (\L_{\pk}),  \pi_{\pk-1} (\q^+_\pk)).
 \end{align*}
 From the geometry of the right triangle
$\triangle (\zero, \L_{\pk}, \q^0_\pk) \subset T_\pk$,
 \begin{align*}
\dist  (\pi_{\pk-1} (\L_{\pk}), \pi_{\pk-1} (\q^+_\pk))
 & \le  \dist(\zero,  \pi_{\pk-1} ( \p_\pk)).
\end{align*}
Recalling that $\p_\pkp$ is parallel to $\pi_{\pk-1} (\p_\pk)$,
combining the preceding four displayed statements with \eqref{E:cn} yields
$\p_\pkp \in \ctk \pi_{\pk-1} (R_\pk) = R_\pkp$ as desired.
This completes the proof of Claim~\ref{C:Akp}.
\end{proof}

{\blue  
\subsubsection*{Verification of \Bkp}
}
\begin{claim}\label{C:Bkp}
Definitions \eqref{E:choose L}--\eqref{E:choose p} imply \Bkp.
\end{claim}

\begin{proof}
Since $\y_\pkp \in [\zero, \p_\pkp]$, from similarity
\begin{align*}
\frac{ \dist(\y_\pkp, P_{\pk-1} ) }{\dist(\zero, P_{\pk-1})} = \frac{\dist (\y_\pkp, \p_\pkp ) }{\dist(\zero, \p_\pkp)}.
\end{align*}
To bound the latter:
\begin{align*}
& \dist (\y_\pkp, \p_\pkp )\\
& = \dist(\pi_{\pk-1}( \y_\pk), \pi_{\pk-1} (\p_\pk) ) + \dist ( \pi_{\pk-1}(\p_\pk), \p_\pkp)\\
& \le \dtk \dist(\zero, \pi_{\pk-1} (\p_\pk) ) + \dist ( \pi_{\pk-1}(\p_\pk), \p_\pkp) & \hbox{ (by \Bk  and similarity)}\\
& \le  \dtk \dist(\zero, \p_\pkp) +  \dist(\pi_{\pk-1} (\p_\pk), \p_\pkp ) & \hbox{(by $[\zero, \pi_{\pk-1} (\p_\pk)] \subset [\zero, \p_\pkp]$)} \\
& \le  \dtk \dist(\zero, \p_\pkp) +  \dist(\pi_{\pk-1} (\q_\pk^+), \p_\pkp ) & \hbox{(by $[\pi_{\pk-1} (\p_\pk),\p_\pkp ] \subset[\pi_{\pk-1} (\q^+_\pk ), \p_\pkp ]$)} \\
& \le \dtk \dist(\zero, \p_\pkp) + 2 \rtk  \dist(\pi_{\pk-1} (\L_{\pk}), \pi_{\pk-1} (\q^0_\pk) )
&\hbox{(by \eqref{E:similarity l})}\\
&\le   (\dtk  + 2\rtk)  \dist(\zero, \p_\pkp). &\hbox{(by $[ \pi_{\pk-1} (\L_{\pk}), \pi_{\pk-1} (\q^0_\pk) ] \subset [\zero, \p_\pkp]$)}
\end{align*}
Therefore,  by \eqref{E:dtk rtk ratio},
\begin{align*}
\frac{ \dist(\y_\pkp, P_{\pk-1} ) }{\dist(\zero, P_{\pk-1})} \le \dtkp
\end{align*}
which is the desired result. This completes the proof of Claim~\ref{C:Bkp}.
\end{proof}

We have shown that  the assumptions \Akp, \Bkp and \Ckp are satisfied, therefore we can continue
the recursion until we arrive at {\bf Case I}, where we get the desired result {\blue \eqref{E:key reformulation}},
or else, at worst, arrive at the following scenario.

\bigskip
{\blue 
\subsection{Final remaining possibility: the recursion reaches $\pk=1$.}
}
Suppose that this recursive procedure does not stop before we find
$\p_1,P_1 \in \R^1$ satisfying
\Aone, \Bone and \Cone and decrease $\pk$ from $2$ to $1$.
We now show the desired result can be established in this case.
Writing $\p_1=p_1$ to emphasize that we are now dealing with 1-tuples,
\Aone yields $P_1 = \{ p_1\} \subset R_1 = \{x \in \R \mid |x| \le \ctk^{n-1} a^1 \}$,
 so $2\dist(0, P_1) \le \diam(P_1^\perp\cap R_1).$
Therefore,
\begin{align*}
\frac{\dist(\y_1, P_1)}{\diam(P_1^\perp\cap R_1)} & \le  \frac{\dist(\y_1, P_1)}{2\dist(0, P_1) }  \\
& \le {\dtn}/{2} \qquad \hbox{ (by \Bone)}.
\end{align*}
Since \Cone guarantees $\tilde P_1$ is disjoint from the interior of $S$, we can
argue exactly as {\bf Case I} to show the supporting hyperplane $P$ of $S$
parallel to $\tilde P_1$ satisfies
\begin{align*}
\frac{\dist(\y, P)}{\diam(P^\perp\cap S)}& \leq n^{3/2} \frac{\dist(\y, \tilde P_1)}{\diam(\tilde P_1^\perp\cap R_n)}\\
&= n^{3/2} \ctk^{n-1} \frac{\dist(\y_1, P_1)}{\diam(P_1^\perp\cap R_1)}\\
& \le n^{3/2}  \ctk^{n-1}  {\blue   \dtn/2 } \\
&= n^{3/2} \ctk^{n-1} (n-{\textstyle \frac{1}{2}}) s^{1/2^{n-1}},
\end{align*}
the desired result {\blue \eqref{E:key reformulation}}.


\bigskip

\bibliographystyle{plain}

\end{document}